\documentclass[a4paper,10pt]{article}
\usepackage[utf8]{inputenc}
\usepackage[british]{babel}

\usepackage{amsthm} %proofs
\usepackage{verbatim} %for comments
\usepackage{graphics,graphicx} % Packages to allow inclusion of graphics
\usepackage{pict2e}	%seems to work while (e)epic doesn't
\usepackage{hyphenat}

\newcommand\p{\circle*{0.3}} %command for putting vertices

\newcommand\nt{\mathcal{NT}}
\newcommand\ent{\nt^{even}}
\newcommand\ntn[1]{\nt_{#1}}
\newcommand\entn[1]{\nt_{#1}^{even}}

\newtheorem{theorem}{Theorem}
\newtheorem{lemma}{Lemma}
\newtheorem{corollary}{Corollary}

\title{Colourability and word-representability of near-triangulations}
\author{Marc Elliot Glen}

\begin{document}
\hyphenation{rep-res-ent-able rep-res-ent-ability poly-omino poly-ominoes
tri-an-gu-la-tion tri-an-gu-la-tions col-our-abil-ity}

\maketitle

\begin{abstract}
A graph $G = (V,E)$ is word-representable if there is a word $w$ over the
alphabet $V$ such that $x$ and $y$ alternate in $w$ if and only if the edge $(x,
y)$ is in $G$. It is known~\cite{HKP2015} that all $3$-colourable graphs are
word-representable, while among those with a higher chromatic number some are
word-representable while others are not.

There has been some recent research on the word-representability of polyomino
triangulations. Akrobotu et al.~\cite{Akrobotu} showed that a triangulation of a
convex polyomino is word-representable if and only if it is $3$-colourable; and
Glen and Kitaev~\cite{GK}
%showed that the same result applies
extended this result to the case of
a rectangular polyomino triangulation when a single domino tile is allowed.

\begin{comment}
This paper
%introduces a new 
studies a 
class of graphs called near\hyp{}triangulations,
which is a grand generalization of polyomino triangulations. A
near\hyp{}triangulation is
% ^overfull hbox ;)
a planar graph in which each
bounded face is a triangle.
%vertex is either at the boundary of the outer face or the centrum of an induced wheel graph 
%(and each bounded face is triangulated; this is the old way of defining it from when they were called ``composite wheel graphs'')
We show that such a graph is $3$-colourable if and only if 
%it contains no induced odd wheels.
%all its centrums have even degree.
each induced wheel graph in it is even.
As a corollary to this result, we obtain a far-reaching generalization of the
previous results by Akrobotu et al. and Glen and Kitaev, which allows placement
on polyominoes of not just any number of arbitrary horizontal or vertical
dominoes, but also any number of arbitrary $n$-ominoes. From these results it
follows that the polyominoes studied by Akrobotu et al. need not be convex to
obtain
%the results;
their results; %?
it is sufficient that they avoid ``internal holes''.
\end{comment}

It was shown in~\cite{DKK} that a near-triangulation is $3$-colourable if and only
if it is internally even. This paper provides a much shorter and more elegant
proof of this fact, and also shows that near-triangulations are in fact a generalization
of the polyomino triangulations studied in~\cite{Akrobotu} and~\cite{GK}, and so
we generalize the results of these two papers,
and solve all open problems stated in~\cite{GK}.

\begin{comment}
In this paper we show that a known result on the $3$-colourability of
near-triangulations
helps to %?
generalize the results in~\cite{Akrobotu} and~\cite{GK}, and solves all open
problems stated in~\cite{GK}. We also provide a new proof for $3$-colourability
of near-triangulations.
\end{comment}
\end{abstract}

\section{Introduction}

% Suppose that $w$ is a word and $x$ and $y$ are two distinct letters in $w$.
%We say that $x$ and $y$ {\it alternate} in $w$ if the deletion of all other
%letters from the word $w$ results in either $xyxy\cdots$ or $yxyx\cdots$.
% ^ maybe leave out

A graph $G=(V,E)$ is called {\em word-representable} if there exists a word $w$
over the alphabet $V$ such that letters $x$ and $y$ alternate in $w$ if and only
if $(x,y)$ is an edge in $E$. For example, the cycle graph $C_4$ on 4 vertices
labelled by 1, 2, 3 and 4 consecutively can be represented by the word $14213243
$.
%There is a long line of research on word-representable graphs, which is
%summarized in the book~\cite{KL}. Two more papers~\cite{CKS,CKS2} closely
%related to this paper have appeared recently.

%This paragraph is on motivaion for studying WR-graphs, which the reviewer mentioned
%(from paper w/ Artem)
There is a long line of research on word-representable graphs, which is
summarized in the book~\cite{KL}. The roots of the theory of
word-representable graphs are in the study of the celebrated Perkins
semigroup~\cite{KS} which has played a central role in semigroup theory since
1960, particularly as a source of examples and counterexamples. Two more
papers~\cite{CKS,CKS2}, closely related to this paper's study of
word-representability of graphs with triangulated faces, have appeared recently.

A graph is $k$-colourable if its vertices can be coloured in at most $k$ colours
so that no pair of vertices with the same colour is connected by an edge.

\begin{theorem}[\cite{HKP2015}]\label{thm:3col} 
All $3$-colourable graphs are word-representable.
\end{theorem}

A {\em wheel} $W_n$ is obtained from the cycle $C_n$ by adding one
%all-adjacent <- the reviewer said this is not very standard
universal
vertex. See Figure~\ref{wheel} for $W_5$.

\begin{comment}
%We note that, for $k\geq 4$, there are examples of non-word-representable graphs
%that are $k$-colourable, but not 3-colourable. For example, the wheel $W_5$ on 6
%vertices is such a graph.

%As examples of $3$-colourable and non-$3$-colourable graphs, all even wheel
%graphs $W_2n$ on $2n + 1$ vertices for $n \geq 1$ are $3$-colourable, while all
%odd wheel graphs are not $3$-colourable. It follows from~\cite{KP} that all odd
%wheel graphs except for $W_5$ are non-word-representable.

%Note that, for $k\geq 4$, there are examples of non-word-representable graphs
%that are $k$-colourable, but not 3-colourable. For example, following
%from~\cite{KP} all odd wheel graphs $W_{2n+1}$ on $2n+2$ vertices for $n \geq 2$
%are such graphs. By contrast, all even wheels are $3$-colourable and
%word-representable, while the only non-$3$-colourable word-representable wheel
%graph is $W_3$.
\end{comment}
Note that, for $k\geq 4$, there are examples of non-word-representable graphs
that are $k$-colourable, but not 3-colourable. For example, following
from~\cite{KP} all wheels $W_n$ for odd $n > 3$ are such graphs; in contrast,
all even wheels are $3$-colourable (and so word-representable). Note that $W_3 =
K_4$ is the complete graph on four vertices, and the only word-representable odd
wheel. $W_5$ is known to be the smallest non-word-representable graph.

%%Reviewer said to remove this line, and just call them wheels and cycles from start
%In this paper, odd (resp., even) wheel graphs are
%called odd (resp., even) wheels.

A {\em near-triangulation} is a
planar
%Reviewer said to use "plane" but planar seems more standard and describes what im talking about
graph in which each inner bounded face is
a triangle (where the outer face may possibly not be a triangle).
In other words, a near-triangulation is a generalization of a triangulation
that may have one untriangulated face.
Let $\ntn{n}$ denote the class of near-triangulations on $n$ vertices, and
let $\entn{n}$ denote members of $\ntn{n}$
%in which each induced wheel is even.
which are {\em internally even}, that is, each inner vertex
(%that is,
vertex not incident to the outer face)
has even degree.
%called even near-triangulations.
See Figure~\ref{wheel} for examples of graphs in $\ntn{7}$ and $\entn{7}$.
We let $\nt = \bigcup_{n \geq 0} \ntn{n}$ and
$\ent = \bigcup_{n \geq 0} \entn{n}$.
It is known~\cite{DKK} that a near-triangulation is $3$-colourable if and only
if
it belongs to $\ent$.
%it is internally even.
This result is of special importance in our paper, and we decided to provide a
new and much more elegant
proof of it (see Theorem~\ref{thm:even3col}).

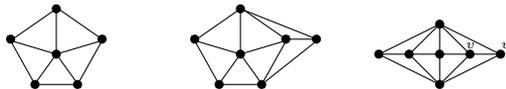
\begin{figure}[h]
 \begin{center}
  \begin{picture}(4,4)
   %\put(0,0){
    \put(0.8,0){\p}	\put(2.2,0){\p}	\put(0,1.5){\p}	\put(3,1.5){\p}	\put(1.5,2.5){\p}
    \put(1.5,1){\p}
    
    \put(0.8,0){\line(1,0){1.4}}
    \put(2.2,0){\line(1,1.8){0.9}}
    \put(3,1.5){\line(-1,0.7){1.4}}
    \put(1.5,2.5){\line(-1,-0.7){1.4}}
    \put(0.8,0){\line(-1,1.8){0.9}}
    
    \put(1.5,1){\line(0,1){1.4}}
    \put(1.5,1){\line(1,0.3){1.4}}
    \put(1.5,1){\line(-1,0.3){1.4}}
    \put(1.5,1){\line(1,-1.4){0.8}}
    \put(1.5,1){\line(-1,-1.4){0.8}}
   %}
  \end{picture}
  %\hfill
  \qquad
  \begin{picture}(5,4)
   %\put(5,0){
    \put(0.8,0){\p}	\put(2.2,0){\p}	\put(0,1.5){\p}	\put(3,1.5){\p}	\put(1.5,2.5){\p}
    \put(1.5,1){\p}
    
    \put(4,1.5){\p}
    
    \put(0.8,0){\line(1,0){1.4}}
    \put(2.2,0){\line(1,1.8){0.9}}
    \put(3,1.5){\line(-1,0.7){1.4}}
    \put(1.5,2.5){\line(-1,-0.7){1.4}}
    \put(0.8,0){\line(-1,1.8){0.9}}
    
    \put(1.5,1){\line(0,1){1.4}}
    \put(1.5,1){\line(1,0.3){1.4}}
    \put(1.5,1){\line(-1,0.3){1.4}}
    \put(1.5,1){\line(1,-1.4){0.8}}
    \put(1.5,1){\line(-1,-1.4){0.8}}
    
    \put(3,1.5){\line(1,0){1}}
    \put(2.2,0){\line(1,0.8){1.8}}
    \put(1.5,2.5){\line(1,-0.4){2.4}}
   %}
  \end{picture}
  %\hfill
  \qquad
  \begin{picture}(4,4)
  % \put(11,0){
    \put(1,0){\p}	\put(0,1){\p}	\put(1,1){\p}	\put(2,1){\p}	\put(1,2){\p}
    
    \put(1,0){\line(0,1){2}}
    \put(-1,1){\line(1,0){4}}
    
    \put(1,0){\line(1,1){1}}
    \put(2,1){\line(-1,1){1}}
    \put(1,2){\line(-1,-1){1}}
    \put(0,1){\line(1,-1){1}}
    
    \put(-1,1){\p}
    \put(3,1){\p}
    
    \put(1,0){\line(2,1){2}}
    \put(1,0){\line(-2,1){2}}
    \put(1,2){\line(2,-1){2}}
    \put(1,2){\line(-2,-1){2}}
    
    \put(1.85,1.2){\tiny $v$}
    \put(3,1.2){\tiny $w$}
   %}
  \end{picture}

 \end{center}
\caption{From left to right: the wheel graph $W_5$, a simple near-triangulation
in $\ntn{7}$ and one in $\entn{7}$. The vertex labelled $v$ is an example of an
inner vertex, and is the centre of a wheel $W_4$, and the vertex labelled $w$
is an example of a boundary vertex.}\label{wheel}
\end{figure}

The motivation for
%defining and
%studying
considering
near-triangulations
%this class of graphs
comes from a subclass of these graphs, called polyomino triangulations, which
have been studied recently by Akrobotu et al.~\cite{Akrobotu} and Glen and
Kitaev~\cite{GK}, and which have had related work dealing with triangulations
and face subdivisions in the recent papers~\cite{CKS,CKS2}. Polyomino
triangulations will be defined in Section~\ref{sec:tri}. Akrobotu et al. showed
that a triangulation of a convex polyomino is word-representable if and only if
it is $3$-colourable, and
%similarly
%Glen and Kitaev showed that a triangulation
%of a rectangular polyomino with a single domino tile is word-representable if
%and only if it is $3$-colourable.
Glen and Kitaev
%showed that the same result applies
extended this result to the case of
a rectangular polyomino triangulation with a single domino.
This paper shows in Theorem~\ref{thm:nt} that a near-triangulation, that avoids
the complete graph $K_4$ as an induced subgraph, is word-representable if and
only if it is $3$-colourable, and so we obtain an elegant generalization of the
%results of both Akrobotu et al. and Glen and Kitaev.
previous results.

\begin{comment}
Note that the problem of determining whether a planar graph is $3$-colourable is
NP-complete~\cite{Stockmeyer}, and so it is unlikely that a nice
characterization of
%$3$-colourability
$3$-colourable graphs exists. Therefore, our Theorem~\ref{thm:nt} is interesting
from the point of view of finding sufficient conditions for $3$-colourability of
planar graphs. For example, a classical result of this type is Gr\"{o}tzsh's
Theorem~\cite{Grotzsh} stating that triangle-free planar graphs are
$3$-colourable.
\end{comment}

This paper is organized as follows. In Section~\ref{sec:nt} we look at
near-triangulations, in particular proving Theorem~\ref{thm:nt} showing that a
near-triangulation that avoids $K_4$ is $3$-colourable if and only if it is
word-representable. Section~\ref{sec:tri} discusses polyomino triangulations as
a specialization of near-triangulations, proving Theorem~\ref{thm:poly} as a
corollary to Theorem~\ref{thm:nt}. Finally, Section~\ref{sec:conc} concludes by
%looking at the
stating a
problem that is left open,
namely that
of the word-representability classification of our graphs that contain $K_4$ as
an induced subgraph.
%, and suggests a conjecture for
%classifying
%these graphs.

\section{Near-triangulations}\label{sec:nt}

The main result of this section is the following:

\begin{theorem}\label{thm:nt}
 A $K_4$-free near-triangulation is $3$-colourable if and only if it is
 word-representable.
\end{theorem}

%Note that every boundary vertex with degree $\ge 2$ in a composite wheel must be
%connected to at least two other boundary vertices.

\begin{lemma}\label{lma}
 %Let $G$ be a graph,
 %\in \ntn{n}$, <-- specification not necessary
 %$x$ be a boundary vertex in $G$ of odd degree,
 Let $x$ be a vertex of odd degree in a graph $G$, and suppose
 %there exists a path $P$, not involving $x$ itself, that connects
 %all of $x$'s neighbours.
 the neighbours of $x$ induce a path $P$.
 %If $G$ is $3$-colourable, then
 In any $3$-colouring of $G$,
%^Reviewer-suggested: above is also accurate but this is more precise as to what is meant
 the end-points of $P$ must have the same colour.
\end{lemma}

\begin{proof}
 Let $a$ and $b$ denote the end-points of $P$. Without loss of generality and
 using colours from $\{1,2,3\}$, assume that $x$ and $a$ have colours $1$ and
 $2$, respectively.
 Starting from $a$ and going along the path, the vertices must be coloured $2,3,
 2,3\dots$. Since the path is of even length, $b$ must be coloured $2$.
\end{proof}

%\begin{conjecture} The class of graphs $nt{n}$ which do not contain $W_3$ is exactly the class of convex polyomino triangulations. \end{conjecture} %no, fig 1. maybe they are
% exactly the same after all

%\begin{theorem}\label{thm:even3col}
% Graphs in the class $\entn{n}$ are $3$-colourable.
%\end{theorem}

The following theorem is proved in~\cite{DKK}, but we provide an alternative
(shorter)
proof of it here.

\begin{theorem}[\cite{DKK}]\label{thm:even3col}
 A graph $G \in \nt$ is $3$-colourable if and only if $G \in \ent$.
\end{theorem}
\begin{proof}
 First, if $G \notin \ent$, then it contains an odd wheel, so obviously it is
 not $3$-colourable as odd wheels are not $3$-colourable.
 
 For the opposite direction, we proceed by induction on the number of vertices,
 with the trivial base case of the single vertex graph. We take a graph $G \in
 \entn{n}$, remove a boundary vertex $v$ from it, and colour the new graph,
 called $G'$, with three colours. We will then show that re-adding $v$ preserves
 $3$-colourability. Note that $G' \in \entn{n-1}$. Using colours from $\{1,2,3\}
 $, it will be shown that in situations in which $v$ has neighbours with all
 three colours, the graph can be recoloured in some way so that $v$ does not
 require colour $4$. If
 %$v$'s neighbours
 the neighbours of $v$
 are coloured with two colours then $v$ can be coloured with the third colour
 and $3$-colourability is preserved.
 
 If
 %$v$'s neighbours
 the neighbours of $v$
 are coloured with three colours, then they can be recoloured with two colours
 as follows. There are two possible situations involving $v$ and its neighbours: 
 $(i)$ there is a path $P$ connecting all of
 %$v$'s neighbours;
 the neighbors of $v$;
 $(ii)$ $v$ is 
 a cut-vertex, meaning that
 %$v$'s neighbours
 the neighbours of $v$
 in $G'$ are in at least two
 %disjoint subsets.
 different connected components.
 %the neighbours of $v$, after $v$'s deletion, contains
 %vertices in at least two disjoint subgraphs. \dots %v is a cut-vertex
 
 For situation $(i)$, going along $P$ starting from one of its end-points, take
 the first instance in which there are $3$ consecutive vertices with different
 colours. Assume without loss of generality that they are coloured $1,2,3$, and
 call the vertices $v_1$, $v_2$ and $v_3$, respectively.
 Vertex $v_2$ must be at the boundary, otherwise Lemma~\ref{lma} can be used to
 show that $v_1$ and $v_3$ must have the same colour (in $G$, $v_2$ must be the
 centre %reviewer suggested centre over centrum
 of an even wheel, so in $G'$ it must have odd degree with a path
 connecting all its neighbours, with $v_1$ and $v_3$ being its end-points).
 Because $v_2$ is at the boundary, it is a cut-vertex in $G'$ meaning that
 removing it increases the number of connected components, each of which is a
 graph in $\ent$.
 Take the component containing $v_3$ and recolour its vertices, swapping $1$ and
 $3$, keeping $2$ the same. Thus the vertices $v_1$, $v_2$ and $v_3$ are now
 coloured $1,2,1$ respectively.
 One can continue going through $P$'s as yet unvisited vertices, looking for
 more examples of three consecutive vertices with three different colours; when
 detecting such vertices, apply the re-colouring argument again. Once all of
 $P$'s vertices are visited, they will be coloured in two colours, so $v$ can be
 coloured with the third colour.
 
 For situation $(ii)$,
 \begin{comment}simply swap $3$s with $2$s within any of the disjoint
 subgraphs that contains a neighbour of $v$ that is coloured $3$, or, if there
 is a neighbour of $v$ coloured $2$ in the same set, swap $3$ with $1$ in that
 set. If a set has three neighbours of $v$ with $3$ different colours, then use
 the method for situation $(i)$ within that set. This guarantees that
 %$v$'s neighbours
 the neighbours of $v$
 have at most two colours and $v$ can be coloured with the third.
 \end{comment}
 the recolouring argument from $(i)$ can be applied to each connected component
 of $G'$ to guarantee that
 %$v$'s neighbours
 the neighbours of $v$
 have at most two colours and $v$ can be coloured with the third.
\end{proof}

For the following theorem, recall that it follows from \cite{KP} that all odd
wheels $W_{2n+1}$ for $n \ge 2$ are non\hyp{}word\hyp{}representable,
%with the only
%word-representable odd wheel being $W_3$.
while $W_3 = K_4$ is word\hyp{}representable. % <- overfull hbox ;)

\begin{theorem}\label{thm:evenwr}
 A
 $K_4$-free
 graph $G \in \nt$
 %that avoids $K_4$
 is word-representable if and only if
 $G \in \ent$.
\end{theorem} %or redefine \ent to include allowing W_3. or exclude W_3
\begin{proof}
 If $G \in \ent$, then it follows from Theorems~\ref{thm:3col}
 and~\ref{thm:even3col} that it is word-representable.
 
 For the other direction, if $G$ is word-representable then it cannot contain
 any odd wheel as an induced subgraph %(except $W_3$??)
 as such graphs are non-word-representable, so $G \in \ent$.
\end{proof}

Theorem~\ref{thm:nt} now follows from Theorems~\ref{thm:even3col} and
\ref{thm:evenwr}.

Another observation to make as a corollary to Theorems~\ref{thm:nt}
and~\ref{thm:even3col} is the following:

\begin{corollary}\label{cor:perfect}
 A
 $K_4$-free
 near-triangulation
 %that avoids $K_4$
 is word-representable if and only if it is perfect.
\end{corollary}
A graph is {\em perfect} if the chromatic number of each of its induced
subgraphs is equal to the size of the largest clique in that subgraph.
\begin{proof}
 Any graph $\in \ent$ has a maximum clique size of $3$, as it avoids $K_4$ as an
 induced subgraph, and so since
 Theorem~\ref{thm:even3col} states that
 they are $3$-colourable,
 these graphs are perfect.
 Conversely, a near-triangulation $\notin \ent$
 %that avoids $K_4$
 %is $4$-colourable but not $3$-colourable,
 has chromatic number $4$, and if it avoids $K_4$ then it has no clique of size
 $4$, so it is not perfect.
 So a $K_4$-free near-triangulation is $3$-colourable if and only if it is
 perfect.
 
 From this and Theorem~\ref{thm:nt} we have the fact that a
 $K_4$-free near-triangulation is word-representable if and only if it is
 perfect.
\end{proof}

%old versions below
\begin{comment}
A graph is {\em perfect} if the chromatic number of each of its induced
subgraphs is equal to the size of the largest clique in that subgraph. From
the fact any graph $\in \ent$ has a maximum clique size of $3$, and as a
corollary to Theorem~\ref{thm:even3col}, we have the observation that a
near-triangulation is $3$-colourable if and only if it is perfect. And from this
and Theorem~\ref{thm:nt} we have the following:

\begin{corollary}\label{cor:perfect}
 A near-triangulation that avoids $K_4$ is word-representable if and only if it
 is perfect.
\end{corollary}
\end{comment}
\begin{comment}
Another observation to make is that a near-triangulation is $3$-colourable if
and only if it is perfect, which is clear from Theorem~\ref{thm:even3col} and
the fact that any graph $\in \ent $ has a maximum clique size of $3$ (it avoids
$K_4$). And so a near-triangulation that avoids $K_4$ is word-representable if
and only if it is perfect (Although there may or may not be ($4$-colourable)
non-word-representable near-triangulations that are perfect because they contain
$K_4$, creating a clique of size $4$).
\end{comment}

\section{Triangulations of polyominoes}\label{sec:tri}

A {\em polyomino} is a plane geometric figure formed by joining one or more
equal squares edge to edge. Letting corners of squares in a polyomino be
vertices, we can treat polyominoes as graphs. In particular, well-known {\em
grid graphs} are obtained from polyominoes in this way. We are interested in
{\em triangulations} of a polyomino, that is, subdividing each square into two
triangles; Figure~\ref{non-convex} shows an example of a polyomino
triangulation. Note that no triangulation is $2$-colourable---at least three
colours are needed to properly colour a triangulation, while four colours are
always enough to colour any triangulation, as it is a planar graph well-known to
be $4$-colourable
by the 4 Colour Theorem. %suggestion of reviewer

The main result of Akrobotu et al.~\cite{Akrobotu},
Theorem~\ref{main-thm-Akrobotu}, is related to {\em convex polyominoes}. A
polyomino is said to be {\em column convex} if its intersection with any
vertical line is convex (in other words, each column has no holes); and
similarly, a polyomino is said to be {\em row convex} if its intersection with
any horizontal line is convex. Finally, a polyomino is said to be convex if it
is row and column convex.

\begin{theorem}[\cite{Akrobotu}]\label{main-thm-Akrobotu} A triangulation of a
convex polyomino is word-representable if and only if it is $3$-colourable.
\end{theorem}

The main result of Glen and Kitaev~\cite{GK}, Theorem~\ref{main-thm-GK}, is
related to a variation of the problem involving {\em polyominoes with domino
tiles}. Polyominoes are objects formed by $1\times 1$ tiles, so that the induced
graphs in question have only (chordless) cycles of length 4. A generalization of
such graphs is allowing domino ($1\times 2$ or $2\times 1$) tiles to be present
in polyominoes, so that in the respective induced graphs (chordless) cycles of
length 6 would be allowed (in a triangulation of a domino tile, it would be
subdivided into four triangles).

\begin{theorem}[\cite{GK}]\label{main-thm-GK} A triangulation of a rectangular
polyomino with a single domino tile is word-representable if and only if it is
$3$-colourable. \end{theorem}

We consider a generalization of both results. As a generalization of the
%restriction to convex polyominoes,
results of Akrobotu et al.\
we consider polyominoes without {\em internal holes}. An internal hole is
defined as a gap in a polyomino which is bounded on all sides. Notice that this
restriction is much more general than the restriction in
Theorem~\ref{main-thm-Akrobotu} to convex polyominoes, but it still prohibits
any chordless cycles of length greater than $3$. However, the counter-examples
by Akrobotu et al. are unavoidable: that is, a graph $G$ containing internal
holes can result in word-representable graphs even when $G$ is
non-$3$-colourable. It turns out that the exact same arguments can be used as
in~\cite{Akrobotu} while replacing the convex restriction with the internal
holes restriction to instantly obtain a more general result. Note in particular
that the boundary shape of the polyomino is not important. See
Figure~\ref{non-convex} for an example of a non-convex polyomino triangulation
not covered by~\cite{Akrobotu}, but that our general result does cover.

\begin{figure}[h]
 \begin{center}
  \begin{picture}(8.5,10.5)

   \put(0,0){\p}	\put(2,0){\p}	\put(4,0){\p}
   \put(0,2){\p}	\put(2,2){\p}	\put(4,2){\p}	\put(6,2){\p}	\put(8,2){\p}
   \put(0,4){\p}	\put(2,4){\p}	\put(4,4){\p}	\put(6,4){\p}	\put(8,4){\p}
   \put(0,6){\p}	\put(2,6){\p}	\put(4,6){\p}	\put(6,6){\p}	\put(8,6){\p}
   \put(0,8){\p}	\put(2,8){\p}	\put(4,8){\p}	\put(6,8){\p}	\put(8,8){\p}
			\put(2,10){\p}	\put(4,10){\p}	\put(6,10){\p}	\put(8,10){\p}

  \put(0,0){\line(0,1){8}}
  \put(2,0){\line(0,1){10}}
  \put(4,0){\line(0,1){2}}
  \put(4,4){\line(0,1){6}}
  \put(6,2){\line(0,1){8}}
  \put(8,2){\line(0,1){8}}
  
  \put(0,0){\line(1,0){4}}
  \put(0,2){\line(1,0){4}}
  \put(6,2){\line(1,0){2}}
  \put(0,4){\line(1,0){4}}
  \put(6,4){\line(1,0){2}}
  \put(0,6){\line(1,0){8}}
  \put(0,8){\line(1,0){8}}
  \put(2,10){\line(1,0){2}}
  \put(6,10){\line(1,0){2}}
  
  \put(0,0){\line(1,1){2}}
  \put(4,0){\line(-1,1){2}}
  \put(2,2){\line(-1,1){2}}
  \put(6,2){\line(1,1){2}}
  \put(0,4){\line(1,1){2}}
  \put(4,4){\line(-1,1){2}}
  \put(8,4){\line(-1,1){2}}
  \put(0,6){\line(1,1){2}}
  \put(4,6){\line(-1,1){2}}
  \put(6,6){\line(-1,1){2}}
  \put(8,6){\line(-1,1){2}}
  \put(2,8){\line(1,1){2}}
  \put(6,8){\line(1,1){2}}
   
  \end{picture}

 \end{center}
\caption{A non-convex polyomino triangulation.}\label{non-convex}
\end{figure}
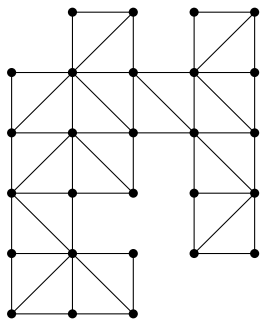

As a generalization of Glen and Kitaev's result, we consider polyominoes in
which any number of $n$-omino tiles are allowed. We call these shapes {\em
polyominoes with $n$-omino tiles}. Additionally, we consider polyominoes that
are not necessarily rectangular; as stated above, the boundary shape of the
polyomino is unimportant.
See Figure~\ref{graphs-with-n-ominoes} for an example of a polyomino with
$n$-omino tiles, without internal holes, and one of its triangulations.

\begin{figure}[h]
\begin{center}
\begin{picture}(20.5,8.5)

\put(0,0){

%\put(0,0){\p} \put(0,2){\p} \put(0,4){\p} \put(0,6){\p} 
%\put(2,0){\p} \put(2,2){\p} \put(2,4){\p} \put(2,6){\p} 
%\put(4,0){\p} \put(4,2){\p} \put(4,4){\p} \put(4,6){\p} 
%\put(6,0){\p} \put(6,2){\p} \put(6,4){\p} \put(6,6){\p} 
%\put(8,0){\p} \put(8,2){\p} \put(8,4){\p} \put(8,6){\p} 

\put(0,0){\line(1,0){8}}
\put(0,4){\line(1,0){8}}
\put(2,6){\line(1,0){6}}
\put(2,2){\line(1,0){6}}
\put(0,8){\line(1,0){2}}

\put(0,0){\line(0,1){8}}
\put(2,0){\line(0,1){4}}
\put(4,0){\line(0,1){6}}
\put(8,0){\line(0,1){2}}
\put(8,4){\line(0,1){2}}
\put(6,2){\line(0,1){4}}
\put(2,6){\line(0,1){2}}

}

\put(12,0){

\put(0,0){\p} \put(0,2){\p} \put(0,4){\p} \put(0,6){\p} 
\put(2,0){\p} \put(2,2){\p} \put(2,4){\p} \put(2,6){\p} 
\put(4,0){\p} \put(4,2){\p} \put(4,4){\p} \put(4,6){\p} 
\put(6,0){\p} \put(6,2){\p} \put(6,4){\p} \put(6,6){\p} 
\put(8,0){\p} \put(8,2){\p} \put(8,4){\p} \put(8,6){\p} 
\put(0,8){\p} \put(2,8){\p}

\put(0,0){\line(1,0){8}}
\put(0,4){\line(1,0){8}}
\put(2,6){\line(1,0){6}}
\put(2,2){\line(1,0){6}}
\put(0,8){\line(1,0){2}}

\put(0,0){\line(0,1){8}}
\put(2,0){\line(0,1){4}}
\put(4,0){\line(0,1){6}}
\put(8,0){\line(0,1){2}}
\put(8,4){\line(0,1){2}}
\put(6,2){\line(0,1){4}}
\put(2,6){\line(0,1){2}}

\put(0,0){\line(1,1){2}}
\put(0,0){\line(1,2){2}}
\put(0,2){\line(1,1){2}}

\put(2,2){\line(1,1){4}}
\put(4,2){\line(1,1){4}}
\put(0,4){\line(1,1){2}}
\put(2,4){\line(1,1){2}}
\put(2,2){\line(1,-1){2}}
\put(4,2){\line(1,-1){2}}
\put(6,2){\line(1,-1){2}}
\put(4,2){\line(2,-1){4}}
%\put(6,4){\line(1,-1){2}}
\put(0,4){\line(2,1){4}}
\put(0,4){\line(.5,1){2}}
\put(0,6){\line(1,1){2}}

}

\end{picture}
\end{center}
\caption{A polyomino with $n$-omino tiles (to the left) and one of its
triangulations (to the right); notice that it is non-word-representable as it
contains $W_7$ as an induced subgraph.}\label{graphs-with-n-ominoes}
\end{figure}
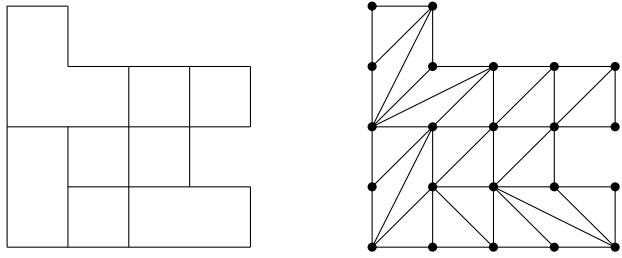
The problem then is in finding a nice characterization of triangulations of such
graphs that are word-representable, similar to Theorems~\ref{main-thm-Akrobotu}
and~\ref{main-thm-GK}. However, allowing arbitrary $n$-ominoes opens up the
possibility of having $W_3 = K_4$ as an induced subgraph, which as stated above
is word-representable, but it is not $3$-colourable; see Figure~\ref{w3-in-poly}
for the smallest such graph, containing a tromino and a square tile.
%Note that this particular graph is non word-representable.

\begin{figure}[h]
 \begin{center}
 \begin{picture}(4.5,4.5)
\put(0,4){\p} \put(2,4){\p} \put(4,4){\p} 
\put(0,2){\p} \put(2,2){\p} \put(4,2){\p} 
\put(0,0){\p} \put(2,0){\p} \put(4,0){\p} 

\put(0,0){\line(1,0){4}}
\put(2,2){\line(1,0){2}}
\put(0,4){\line(1,0){4}}
\put(0,0){\line(0,1){4}}
\put(2,2){\line(0,1){2}}
\put(4,0){\line(0,1){4}}

\put(0,0){\line(1,1){2.1}}
\put(0,2){\line(1,1){2.1}}
\put(2,0){\line(1,1){2.1}}
\put(2,4){\line(1,-1){2.1}}

\put(0,0){\line(1,2){2}}
\put(0,0){\line(2,1){4}}
 \end{picture}
 \end{center}
 \caption{The smallest
 %word-representable and
 %non-$3$-colourable
 polyomino triangulation with $n$-omino tiles
 containing $K_4$.}\label{w3-in-poly}
\end{figure}
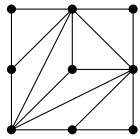

Notice that the appearance of $K_4$ here is because of the inner ``bend'' in the
tromino, as such a bend is the only situation in a polyomino triangulation in
which an inner vertex may have degree $3$. From this it is clear that the
result for polyominoes with $n$-omino tiles cannot be as elegant as the other
results, as there exist non-$3$-colourable graphs that may
%still
or may not
be word-representable. Therefore the main result of this section is the
following:
%, which uses the alternative name for $W_3$, namely the complete graph $K_4$.
%necessary? I might not need
%to consider it as a complete graph K_4

\begin{theorem}\label{thm:poly}
 A
 $K_4$-free
 triangulation of a polyomino with $n$-omino tiles and without internal holes
 %that avoids $K_4$ as an induced subgraph,
 is word-representable if and only if it is $3$-colourable.
\end{theorem}
%change theorem to include W_3-containing graphs?
%Or discuess W_3 separately?

It is clear to see that these types of polyomino triangulations are
near\hyp{}triangulations, with the requirement of no internal holes being equivalent
to the requirement that every bounded face is a triangle. Theorem~\ref{thm:poly}
therefore follows from Theorem~\ref{thm:nt}.

\begin{comment} % relocated to previous section, more relevant there?
Another observation to make is that a composite wheel graph is $3$-colourable if
and only if it is perfect, which is clear from the fact that any graph $\in \ent
$ has a maximum clique size of $3$ (it avoids $K_4$). And so a composite wheel
graph that avoids $K_4$ is word-representable if and only if it is perfect
(Although there may be ($4$-colourable) non-word-representable composite wheels
that are perfect because they contain $K_4$, creating a clique of size $4$).
\end{comment}

\section{Concluding remarks}\label{sec:conc}
It is still left as an open problem the word-representability classification of
near-triangulations,
and graphs in general,
that contain $K_4$ as an induced subgraph. As $K_4$ is non-$3$-colourable, a
very different approach from the one above involving colours is required to find
an elegant classification; since planar graphs are always $4$-colourable, a
classification based on colourability is not possible. %Inspired by
%Corollary~\ref{cor:perfect}, we conjecture, and leave  as an open problem, that
%any near-triangulation, that may contain the complete graph $K_4$ as an induced
%subgraph,
%is word-representable if and only if it is perfect.
% ^ not true, see figure w3-in-poly for a counter-example

\bibliographystyle{../mystyle}
\bibliography{../ref}
\end{document}